\documentclass[12pt]{amsart}

\usepackage{setspace}

\usepackage{amsmath}
\usepackage{amstext}
\usepackage{amsthm}
\usepackage{amssymb}
\usepackage[a4paper, width=6in, height=9in, centering]{geometry}
\usepackage{enumerate}
\usepackage{xcolor}
\usepackage{comment}
\usepackage{hyperref}
\usepackage{graphicx}

\newtheorem{theorem}{Theorem}[section]
\newtheorem{lemma}[theorem]{Lemma}
\newtheorem{proposition}[theorem]{Proposition}
\newtheorem{corollary}[theorem]{Corollary}
\newtheorem{remark}[theorem]{Remark}

\newtheorem{definition}[theorem]{Definition}

\DeclareMathOperator{\graph}{graph}

\DeclareMathOperator{\Div}{div}

\begin{document}

\title{Translators Asymptotic to Planes}

\author{Stephen Lynch and Giuseppe Tinaglia}
\address{King's College London, Mathematics, London, U.K.}

\begin{abstract} We prove that a vertical plane is the only complete translator, properly immersed in $\mathbb{R}^3$ and having finite topology, whose ends are asymptotic to vertical planes.
\end{abstract}
\maketitle

\section{Introduction}
Translators arise naturally in the study of mean curvature flow as self-similar solutions that move by translation. Their geometric and analytic properties have been studied extensively, see for instance~\cite[Chapter 13]{book-EGF}, mainly because of their role in understanding type-II singularities of mean curvature flow \cite{Hamilton, HS}.

Our main theorem classifies complete properly immersed translators with finite topology whose ends are asymptotic to vertical planes, see Definition~\ref{def: asymptotic}.

\begin{theorem}\label{uniqueness}
Let $\Sigma$ be a complete translator properly immersed in $\mathbb{R}^3$ with finite topology whose ends are asymptotic to vertical planes. Then $\Sigma$ is a vertical plane. 
\end{theorem}

We note that since $\Sigma$ is not assumed to be embedded in Theorem~\ref{uniqueness}, its asymptotic vertical planes need not be parallel. 

Our study was inspired by Khan's work on complete embedded translators with finite total curvature~\cite{Khan}, which, among other things, shows that such surfaces satisfy the hypotheses of Theorem~\ref{uniqueness}. We therefore have the following corollary of Theorem~\ref{uniqueness}. 

\begin{corollary}
Every complete translator embedded in $\mathbb{R}^3$ with finite total curvature is a vertical plane.
\end{corollary}

Our proof of Theorem~\ref{uniqueness} relies principally on an identity expressing the total mean curvature of any smooth compact translator as the flux of $e_3^\top$ across its boundary, see Lemma~\ref{lem: total mean curvature}.

\begin{remark}
After originally posting this article we learned of the work \cite{Neves--Tian}, by Neves and Tian, concerning the classification of Lagrangian translators. Some elements of our proof are similar to theirs---in particular, they derive the same flux formula for the total mean curvature and utilise it to prove flatness in their setting. 
\end{remark}

\textbf{Acknowledgements.} We are grateful to Felix Schulze for bringing to our attention the reference \cite{Neves--Tian}. 

\section{Preliminaries}

A translator $\Sigma \subset \mathbb{R}^3$ is a smooth surface whose mean curvature vector $H$ satisfies 
    \[H = e_3^\perp.\]
Equivalently, the family of surfaces $\Sigma + t e_3$ for $t \in \mathbb{R}$ is a solution to the mean curvature flow.

In the following lemma, we express the total mean curvature of any smooth compact translator as the flux of $e_3^\top$ across its boundary.

\begin{lemma}\label{lem: total mean curvature}
For a smooth compact translator $\Sigma$ with boundary $\partial\Sigma$ we have
    \begin{equation}\label{total mean curvature}
        \int_{\Sigma} |H|^2 = \int_{\partial \Sigma} \langle \eta, e_3^\top\rangle,
    \end{equation}
where $\eta$ is the outward conormal to $\partial \Sigma$.   
\end{lemma}
\begin{proof}
Integrate the identity 
    \[\Div_\Sigma e_3^\top = -\Div_\Sigma e_3^\perp = \langle H, e_3^\perp \rangle = |H|^2\]
and apply the divergence theorem. 
\end{proof}

Finally, we  define what it means for an end to be asymptotic to a vertical plane.

\begin{definition}\label{def: asymptotic}
Let $\Sigma_i$ be an end of a translator $\Sigma$, i.e. a component of $\Sigma \setminus B_r(0)$ for some sufficiently large radius $r$. We say that $\Sigma_i$ is asymptotic to a vertical plane if there exists a vertical plane $\Pi_i$ such that the following holds:
   
$\Sigma_i = \graph u_i$ for a function $u_i : \Pi_i \setminus \Omega_i \to \mathbb{R}$, where $\Omega_i \subset  \Pi_i$ is open and bounded, and $\|u_i\|_{C^1(\Pi_i \setminus B_R(0))} \to 0$ as $R \to \infty$. 
\end{definition}

\section{Proof of Theorem~\ref{uniqueness}}
To prove Theorem~\ref{uniqueness} we produce an exhaustion of $\Sigma$ by smooth, compact regions such that the boundary integral in \eqref{total mean curvature} converges to zero. This implies that $H$ vanishes identically and hence $\Sigma$ is a vertical plane.

We first prove a variant of Theorem~\ref{uniqueness} which imposes a slightly different hypothesis on the ends. 

\begin{proposition}\label{prop: strongly planar}
Let $\Sigma$ be a complete translator properly immersed in $\mathbb{R}^3$. Suppose there exists a radius $r > 0$ such that:
    \begin{itemize}
        \item $\Sigma \setminus B_r(0)$ has finitely many components, denoted $\Sigma_i$;
        \item each component $\Sigma_i = \graph u_i$ for a function $u_i : \Pi_i \setminus \Omega_i \to \mathbb{R}$, where $\Omega_i \subset  \Pi_i$ is open and bounded, and   
        \begin{equation}\label{eq:strong decay}
        \lim_{R \to \infty} \int_{\Pi_i \, \cap \, \partial B_R(p_i)} |Du_i|^2 = 0
        \end{equation}
        for some $p_i \in \Pi_i$. 
    \end{itemize}
Then $\Sigma$ is a vertical plane. 
\end{proposition}
\begin{proof}
For each $\Sigma_i$, up to an ambient isometry we may assume $\Pi_i = \{x_1 = 0\}$ and $p_i = 0$. We introduce polar coordinates for $\Pi_i$, conflating $(0, r\cos\theta,r\sin\theta)$ with $(r,\theta)$. We may then express points in $\Sigma_i$ as $(u_i(r,\theta), r, \theta)$. For all sufficiently large $R$ we define a map
    \[\gamma_{i,R}(\theta) := (u_i(R, \theta), R, \theta),\]
giving rise to a smoothly embedded closed curve
    \[\Gamma_{i,R} := \gamma_{i,R}(S^1)\subset \Sigma_i.\]
Since there are only finitely many components $\Sigma_i$ and $\Sigma$ is properly immersed, for each large $R$ the collection of curves $\Gamma_{i,R}$ bounds a compact region in $\Sigma$, which we denote by $\Sigma_R$. We will demonstrate that 
    \[\lim_{R \to \infty} \int_{\Gamma_{i,R}} \langle\eta,e_3^\top\rangle =  0\]
for each $i$, where $\eta$ is the outward unit conormal to $\partial \Sigma_R$. The claim then follows, since Lemma~\ref{lem: total mean curvature} implies that 
    \[\lim_{R\to\infty} \int_{\Sigma_R} |H|^2 = 0\]
and hence $H$ vanishes everywhere, meaning $\Sigma$ can only be a vertical plane. 

For the following computation let $i$ and $R$ be fixed and to ease notation let us write $\gamma$ instead of $\gamma_{i,R}$, $\Gamma$ instead of $\Gamma_{i, R}$, etc. The outward conormal to $\Gamma$ is  
    \[\eta(R,\theta) = \frac{\gamma'(\theta)}{|\gamma'(\theta)|} \times \nu(R, \theta).\]
where $\nu$ is the upward unit normal to $\Sigma$, given by
    \[\nu = \frac{e_1 - Du}{\sqrt{1+|Du|^2}}.\]
We therefore have
    \begin{align*}
    \int_{\Gamma} \langle \eta, e_3^\top\rangle &= \int_0^{2\pi} \left\langle \frac{\gamma'(\theta)}{|\gamma'(\theta)|} \times \nu(R, \theta), e_3\right\rangle |\gamma'(\theta)|\,d\theta\\
    &= \int_0^{2\pi} \langle \gamma'(\theta) \times \nu(R, \theta), e_3\rangle\,d\theta.
    \end{align*}
Evaluating the integrand at $(R,\theta)$ gives
    \begin{align*} 
    \langle \gamma'(\theta) \times \nu, e_3\rangle &= \langle\gamma'(\theta),e_1\rangle \langle\nu,e_2\rangle - \langle\gamma'(\theta),e_2\rangle \langle\nu, e_1\rangle\\
    &= \frac{1}{\sqrt{1+|Du|^2}}(R\sin\theta|D_2 u|^2 - R\cos\theta D_3 u D_2 u + R\sin\theta)\\
    &= R\sin\theta + R \cdot O(|Du(R, \theta)|^2),
    \end{align*}
and hence we have
    \begin{align*}
    \int_{\Gamma} \langle \eta, e_3^\top\rangle &= \int_0^{2\pi} R\sin\theta\,d\theta + \int_0^{2\pi} R\cdot O(|Du(R, \theta)|^2)\,d\theta\\
    &= \int_0^{2\pi} R\cdot O(|Du(R, \theta)|^2) \,d\theta.
    \end{align*}
The hypothesis \eqref{eq:strong decay}, expressed in polar coordinates, ensures that the right hand side tends to zero as $R \to \infty$. This completes the proof. 
\end{proof}

Our aim now is to prove that any end which is asymptotic to a vertical plane automatically satisfies \eqref{eq:strong decay}. This will make use of the following interior gradient estimate, which is obtained by a standard application of Bernstein's technique.\footnote{We note that Gama, Mart\'{i}n and M{\o}ller have independently obtained the same estimate by different arguments---it is implied by \cite[Proposition~3.10]{GMM} via a standard interpolation inequality.}

\begin{lemma}\label{lem: grad est}
Let $\Sigma$ be a translator of the form $\graph u$, where $u : \Pi\cap \overline{B_R(p)} \to \mathbb{R}$, $\Pi$ is a vertical plane containing the point $p$, and we assume $R \geq 1$. If $|Du|^2 \leq \frac{1}{4}$ holds everywhere in $\Pi\cap \overline{B_R(p)}$ then we have 
    \[|Du(p)|^2 \leq \frac{400}{R}\max_{\Pi \,\cap \,\partial B_R(p)} u^2.\]
\end{lemma}
\begin{proof}
Up to an ambient isometry we may assume $\Pi = \{x^1 = 0\}$. A unit normal to $\Sigma$ is given by 
    \[\nu := \frac{e_1-Du}{\sqrt{1+|Du|^2}},\]
and we have
    \[\langle H,\nu\rangle = \bigg(\delta^{ij} - \frac{D^i u D^j u}{1+|Du|^2}\bigg)\frac{D_i D_j u}{\sqrt{1+|Du|^2}} =: a^{ij} \frac{D_i D_j u}{\sqrt{1+|Du|^2}}\]
and 
    \[\langle e_3, \nu\rangle = - \frac{D_3 u}{\sqrt{1+|Du|^2}}.\]
Therefore, the translator condition for $\Sigma$ is equivalent to 
    \[a^{ij} D_i D_j u + D_3 u = 0.\]
Here and throughout the proof indices range over $\{2,3\}$. 

Let us write $L := a^{ij}D_i D_j + D_3$. We then have
    \begin{align*}
    L \frac{u^2}{2} = a^{ij} D_i u D_j u =\frac{|Du|^2}{1+|Du|^2}.
    \end{align*}
and, using the Cauchy--Schwarz inequality,
    \begin{align*}
    L\frac{|Du|^2}{2} &=-D^k u D_k a^{ij} D_i D_j u + a^{ij} D_i D^k u D_j D_k u \geq \bigg(1 - 5\frac{|Du|^2}{1+|Du|^2}\bigg)|D^2 u|^2.
    \end{align*}
Let $\varphi \geq 0$ be a cutoff function and $\Lambda > 0$ a constant, both to be chosen later. From the above computations we see that the auxiliary quantity $G := \varphi^2 \frac{|Du|^2}{2} + \Lambda \frac{u^2}{2}$ satisfies
    \begin{align*}
    L G &= \frac{|Du|^2}{2}L\varphi^2 + a^{ij} D_i \varphi^2 D_j |Du|^2 + \varphi^2 L \frac{|Du|^2}{2} + \Lambda L\frac{u^2}{2}\\
    &\geq \frac{|Du|^2}{2} L \varphi^2 + a^{ij} D_i \varphi^2 D_j |Du|^2\\
    &\qquad + \bigg(1 - 5\frac{|Du|^2}{1+|Du|^2}\bigg)\varphi^2|D^2 u|^2 + \frac{\Lambda|Du|^2}{1+|Du|^2}.
    \end{align*}
Using $|a| \leq 2$ we estimate
    \[L\varphi^2 \geq -4 \varphi |D^2 \varphi| - 4|D\varphi|^2 - 2 \varphi|D_3\varphi|\]
and 
    \begin{align*}
    a^{ij} D_i \varphi^2 D_j |Du|^2 &\geq -8\varphi|D\varphi||Du||D^2 u| \geq - \frac{1}{2}\varphi^2|D^2 u|^2- 32|D\varphi|^2|Du|^2,
    \end{align*}
leading to  
    \begin{align*}
    LG &\geq -|Du|^2(2 \varphi |D^2 \varphi| + 34|D\varphi|^2 + \varphi|D_3\varphi|) \\
    &\qquad + \bigg(\frac{1}{2}-5\frac{|Du|^2}{1+|Du|^2}\bigg)\varphi^2 |D^2 u|^2 + \frac{\Lambda|Du|^2}{1+|Du|^2}.
    \end{align*}
Let $\varphi$ be chosen so that $\varphi(p)=1$, $\varphi \equiv 0$ on $\partial B_R(p)$, $0 \leq \varphi \leq 1$ and 
    \[R|D\varphi| \leq 2, \qquad  R^2|D^2\varphi| \leq 8.\]
In addition, let $\Lambda = 400R^{-1}$. Inserting these choices above and using $R \geq 1$, we finally arrive at 
    \begin{align*}
    L G &\geq \bigg(\frac{400}{1+|Du|^2}-200\Bigg)R^{-1}|Du|^2 + \bigg(\frac{1}{2}-5\frac{|Du|^2}{1+|Du|^2}\bigg)\varphi^2 |D^2 u|^2.
    \end{align*}
The right-hand side is nonnegative in $B_R(p)$ since we are assuming $|Du|^2 \leq \frac{1}{4}$, so by the maximum principle
    \begin{align*}
    \frac{|Du(p)|^2}{2} \leq \max_{B_{R}(p)} G = \max_{\partial B_R(p)} G = \frac{400}{R} \max_{\partial B_R
    (p)} \frac{u^2}{2}. 
    \end{align*}
\end{proof}

We now complete the proof of our main theorem.

\begin{proof}[Proof of Theorem~\ref{uniqueness}]
Let $\Sigma \subset \mathbb{R}^3$ be a complete translator satisfying the hypotheses of Theorem~\ref{uniqueness}. We proceed by verifying that each one of the finitely many annular ends satisfies equation \eqref{eq:strong decay}, so that the claim follows from Proposition~\ref{prop: strongly planar}. 

Consider an annular end $\Sigma_i = \graph u_i$ and fix a point $p_i \in \Pi_i$ for reference. Since, by hypothesis, we have 
    \[\sup_{\Pi_i \setminus B_R(p_i)} |Du_i| \to 0 \qquad \text{as} \qquad R \to \infty,\]
we certainly have $|Du_i| \leq 1/4$ in the region $\Pi_i \setminus B_{R/2}(p_i)$ whenever $R$ is sufficiently large. Therefore, given any $q \in \Pi_i \setminus B_R(p_i)$, by Lemma~\ref{lem: grad est} we have 
    \[|Du_i(q)|^2 \leq \frac{800}{R} \max_{\partial B_{R/2}(q)}u_i^2\]
for sufficently large $R$. Since $q \in \Pi_i \setminus B_{R}(p_i)$ was chosen arbitrarily we conclude that 
    \[\sup_{\Pi_i \setminus B_R(p_i)} |Du_i|^2 \leq \frac{800}{R} \sup_{\Pi_i \setminus B_{R/2}(p_i)} u_i^2\]
for sufficiently large $R$. But we also assume 
    \[\sup_{\Pi_i \setminus B_{R/2}(p_i)} u_i^2 \to 0 \qquad \text{as} \qquad R \to \infty,\]
so this shows that 
    \[\sup_{\Pi_i \setminus B_R(p_i)} |Du_i|^2 \leq o(R^{-1})\]
and hence 
    \[\int_{\Pi_i \, \cap \, \partial B_R(p_i)} |Du_i|^2 = \int_0^{2\pi} R|Du_i|^2 \, d\theta = o(1)\]
as $R \to \infty$, which is \eqref{eq:strong decay}. 
\end{proof}

\bibliographystyle{plain}
\bibliography{references}

\end{document}